\documentclass[a4paper,11pt]{article}

\def\ms{\textbf{2020 Mathematics Subject Classifications}: }
\providecommand\key[1]{\par \textbf{Keywords: }#1 \vspace{0,5cm}}
\providecommand\add[2]{\textsc{{#1}} {\newline \indent \textbf{ Email(s): }#2}}

\usepackage{geometry}
\usepackage{amsmath,amsthm,amscd,amsfonts,amssymb,enumerate}
\usepackage{graphicx}		
\usepackage{epsfig}
\usepackage{epstopdf}
\usepackage{float}
\usepackage{microtype}
\usepackage{tikz}
\usetikzlibrary{arrows.meta, positioning}

\usepackage{lineno}

\newtheorem{theorem}{Theorem}[section]

\newtheorem{proposition}[theorem]{Proposition}
\newtheorem{corollary}[theorem]{Corollary}
\newtheorem{definition}[theorem]{Definition}

\begin{document}
\setlength{\parindent}{0pt}
\title{More about cofinally Bourbaki quasi-complete metric spaces}

\author{Argha Ghosh $^{1*}$}

\date{}
\maketitle

\begin{abstract}
We characterize cofinally Bourbaki quasi-complete metric spaces and their completions in terms of certain Lipschitz-type functions. To this end, we introduce and study a new class of functions, namely strongly uniformly locally Lipschitz functions, which lie strictly between Lipschitz functions and uniformly locally Lipschitz functions. We show that a metric space $\langle X, d \rangle$ is cofinally Bourbaki quasi-complete if and only if the class of strongly uniformly locally Lipschitz functions on $\langle X, d \rangle$ coincides with the (a priori) larger class of locally Lipschitz functions. Moreover, the completion of $\langle X, d \rangle$ is cofinally Bourbaki quasi-complete if and only if the class of strongly uniformly locally Lipschitz functions agrees with the class of Cauchy-Lipschitz functions. Finally, we provide several characterizations of cofinally Bourbaki quasi-complete metric spaces and their completions using functions that preserve certain classes of Cauchy-type sequences. \\
\end{abstract}

\ms {Primary 26A16, 54E40, 26A99; Secondary 26A15,
54E50, 54D99.}\\

\key {cofinally Bourbaki quasi-completeness, strongly uniformly locally Lipschitz functions, uniformly star superparacompact spaces, between complete and compact spaces, Cauchy-type sequences }\\


\add {Department of Mathematics, Manipal Institute of Technology, Manipal Academy of Higher Education, Manipal, Karnataka, India $^{1*}$ (corresponding author) 
	\\ 
}
	{argha.ghosh@manipal.edu $^{1*}$ (corresponding author) 
	}

\section{Introduction}

In classical analysis, completeness and compactness are among the most fundamental properties of metric spaces. A metric space is said to be \emph{complete} if every Cauchy sequence converges, and \emph{compact} if every sequence has a convergent subsequence. Although these properties are logically independent, each imposes a strong form of sequential regularity. Nevertheless, there exist metric spaces that are neither complete nor compact but retain structural features reminiscent of these classical notions. Over the past two decades, this observation has led to the development of various intermediate forms of completeness, designed to capture and extend the core ideas of convergence, boundedness, and continuity.

One such prominent generalization is the class of \emph{Atsuji spaces}, also known as \emph{uniform continuity spaces} (UC-spaces). A metric space $\langle X, d \rangle$ is called a UC-space if every real-valued continuous function on $X$ is uniformly continuous. Equivalently, $X$ is a UC-space if and only if every pseudo-Cauchy sequence consisting of distinct terms clusters. A sequence $\langle x_n \rangle$ in $X$ is called \emph{pseudo-Cauchy} if for every $\varepsilon > 0$ and every $n \in \mathbb{N}$, there exist indices $j > k \geq n$ such that $d(x_j, x_k) < \varepsilon$. The theory of UC-spaces, which originated with Doss \cite{Doss1947}, which was further investigated by Nagata \cite{Nagata1950} and Atsuji \cite{Atsuji1958}, has since attracted broad attention. A comprehensive survey of equivalent conditions for a metric space to be UC can be found in \cite{Jain2006}.

Another significant development in this area was the introduction of \emph{cofinally complete} metric spaces by Howes \cite{Howes1971}. A sequence $\langle x_n \rangle$ in $X$ is called \emph{cofinally Cauchy} if for every $\varepsilon > 0$, there exists an infinite subset $N_\varepsilon \subseteq \mathbb{N}$ such that $d(x_n, x_m) < \varepsilon$ for all $n, m \in N_\varepsilon$. The space is cofinally complete if every such sequence has a cluster point. The study of these spaces gained renewed attention through the work of Beer \cite{Beer2008}, provided a variety of new characterizations via  a geometric functional. In a subsequent paper by Garrido \cite{Beer2015}, this approach was further extended via Lipschitz-type mappings.

This line of investigation was further developed by Aggarwal and Kundu \cite{Aggarwal2016}, and subsequently by Gupta and Kundu \cite{Gupta2020, Gupta2021}, who provided characterizations of cofinally complete spaces in terms of geometric functionals and certain classes of mappings. A key concept introduced in this context is that of \emph{cofinally Cauchy regular} (CC-regular) functions—mappings that send cofinally Cauchy sequences to cofinally Cauchy sequences \cite{Aggarwal2017}. For a detailed account of cofinally completeness and related notions, we refer the reader to the monograph \cite{Kundu2023}.

A complementary approach found in several articles focuses on the role of chainable structures in defining and characterizing completeness-type properties of metric spaces. This framework relies not only on the local behavior of the metric but also on how sets and sequences interact with iterated neighborhoods and chains of small steps. A metric space $\langle X, d \rangle$ is said to be \emph{$\varepsilon$-chainable} if any two points in $X$ can be joined by an $\varepsilon$-chain, i.e., a finite sequence $\{x_0, x_1, \ldots, x_n\}$ with $x_0 = x$, $x_n = y$ and $d(x_{i-1}, x_i) < \varepsilon$ for all $i$. The $\varepsilon$-chainable component of $x$, denoted $S^{\infty}_d(x, \varepsilon)$, is defined recursively by $S^1_d(x, \varepsilon) := S_d(x, \varepsilon)$ and $S^{n}_d(x, \varepsilon) := \left(S^{n-1}_d(x, \varepsilon)\right)^\varepsilon$, with $S^{\infty}_d(x, \varepsilon) := \bigcup_{n} S^n_d(x, \varepsilon)$. A subset $A \subseteq X$ is called \emph{Bourbaki bounded} (or finitely chainable) if for every $\varepsilon > 0$, there exist $m \in \mathbb{N}$ and finitely many points $p_1, \ldots, p_k \in X$ such that $A \subseteq \bigcup_{i=1}^k S^m_d(p_i, \varepsilon)$ \cite{BeerBook, Garrido2014, Kundu2017}.

To characterize finite chainable metric spaces sequentially, Garrido and Meroño \cite{Garrido2014} introduced the notions of \emph{Bourbaki--Cauchy} and \emph{cofinally Bourbaki--Cauchy} sequences. A sequence $\langle x_n \rangle$ in $X$ is said to be \emph{Bourbaki--Cauchy} if for every $\varepsilon > 0$, there exist $m \in \mathbb{N}$, $n_0 \in \mathbb{N}$, and $p \in X$ such that
\[
  x_n \in S^m_d(p, \varepsilon) \quad \text{for all } n \geq n_0.
\]
Similarly, $\langle x_n \rangle$ is said to be \emph{cofinally Bourbaki--Cauchy} if for every $\varepsilon > 0$, there exist an infinite set $N_\varepsilon \subseteq \mathbb{N}$, $m \in \mathbb{N}$, and $p \in X$ such that
\[
  x_n \in S^m_d(p, \varepsilon) \quad \text{for all } n \in N_\varepsilon.
\]
A metric space $\langle X, d \rangle$ is called \emph{Bourbaki complete} (respectively, \emph{cofinally Bourbaki complete}) if every Bourbaki--Cauchy (respectively, cofinally Bourbaki--Cauchy) sequence in $X$ has a cluster point.

Further characterizations of Bourbaki complete metric spaces, as well as their completions, were provided by Aggarwal and Kundu \cite{Aggarwal2017M} in terms of mappings that preserve Bourbaki--Cauchy sequences, referred to as \emph{BC-regular functions}, along with certain geometric functionals. A comprehensive treatment of Bourbaki completeness and related notions may be found in the doctoral thesis of Meroño \cite{Thesis}.

In analogy with the theory of cofinally complete metric spaces, Aggarwal and Kundu \cite{Aggarwal2017} provided some characterizations of cofinally Bourbaki complete spaces. This line of investigation was subsequently extended by Gupta and Kundu \cite{Gupta2020}, who developed several additional characterizations, including characterizations of their completions, based on mappings that preserve Bourbaki--Cauchy-type sequences. In particular, Aggarwal and Kundu \cite{Aggarwal2017} introduced and studied the class of \emph{cofinally Bourbaki--Cauchy regular} (CBC-regular) functions---mappings that send cofinally Bourbaki--Cauchy sequences to cofinally Bourbaki--Cauchy sequences. These sequence-preserving mappings play a role analogous to CC-regular functions in the cofinal Cauchy setting and have proven instrumental in capturing structural properties associated with cofinally Bourbaki completeness.

More recently, Adhikary et al. \cite{Adhikary2024} (see also \cite{Das2021}) introduced the notions of \emph{qC-precompact} and \emph{Bourbaki quasi-complete} metric spaces. A sequence is called \emph{Bourbaki quasi-Cauchy} if there exists $n_0 \in \mathbb{N}$ such that for every $\varepsilon > 0$, all terms $x_j, x_k$ with $j, k \geq n_0$ can be joined by an $\varepsilon$-chain. A space is said to be Bourbaki quasi-complete if every such sequence has a cluster point. Among other results, the authors established that a metric space is qC-precompact if and only if every sequence admits a Bourbaki--Cauchy subsequence. Several characterizations were also provided in terms of Lipschitz-type and ward-continuous functions.

In \cite{Das2024}, Das et al. introduced and explored the notion of \emph{cofinally Bourbaki quasi-complete} uniform spaces and metric spaces. In this setting, a sequence $\langle x_n \rangle$ is said to be \emph{cofinally Bourbaki quasi-Cauchy} if for every $\varepsilon > 0$, there exists an infinite subset $N_\varepsilon \subseteq \mathbb{N}$ such that $x_j$ and $x_k$ can be joined by an $\varepsilon$-chain for any $j, k\in N_\varepsilon$. A space is cofinally Bourbaki quasi-complete if every such sequence clusters. Building on this, several equivalent formulations of this completeness notion—using geometric functionals and open covers—were established in \cite{Ghosh2025}.

Despite this progress, unlike the cases of cofinally complete and cofinally Bourbaki complete spaces, the function-theoretic characterization of cofinally Bourbaki quasi-complete metric spaces remains, to date, largely unexplored. In particular, while cofinally complete and cofinally Bourbaki complete spaces have been characterized via Lipschitz-type, CC-regular, and CBC-regular functions, no analogous framework has yet been established for the cofinally Bourbaki quasi-complete setting.

This raises several natural and fundamental questions regarding the structure of such spaces:
\begin{itemize}
    \item Can cofinally Bourbaki quasi-complete metric spaces be characterized by the behavior of locally Lipschitz functions?
    \item Is it possible to identify a class of mappings—intermediate between locally Lipschitz and Lipschitz—that precisely captures this completeness property?
    \item Can the cofinally Bourbaki quasi-completeness of a space be inferred from the behavior of certain classes of functions on cofinally Bourbaki quasi-Cauchy sequences?
    \item How do such function-theoretic characterizations extend to completions of metric spaces?
\end{itemize}

In this paper, we address these questions, to this end we introduce a new class of functions, termed \emph{strongly uniformly locally Lipschitz} (also referred to as \emph{uniformly locally chain-Lipschitz}) functions. A mapping $f : X \to Y$ between metric spaces is said to be strongly uniformly locally Lipschitz if there exists $\delta > 0$ such that for each $x \in X$, the restriction of $f$ to the $\delta$-chainable component $S^\infty_d(x, \delta)$ is Lipschitz. This class lies strictly between the classes of uniformly locally Lipschitz and Lipschitz mappings. Moreover, in chainable metric spaces—such as normed linear spaces—strongly uniformly locally Lipschitz functions coincide with classical Lipschitz functions.

Our main results are as follows. In Section~2, we show that a metric space $\langle X, d \rangle$ is cofinally Bourbaki quasi-complete if and only if every locally Lipschitz function is strongly uniformly locally Lipschitz, assuming $X$ is strongly uniformly locally bounded (see Theorem~\ref{uniformly_lipschitz_firsttheorem}). We also prove that the completion $\langle \widehat{X}, \widehat{d} \rangle$ is cofinally Bourbaki quasi-complete if and only if every Cauchy--Lipschitz function on $X$ is strongly uniformly locally Lipschitz (see Theorem~\ref{completeion_uniformly_lipschitz_theorem}). In addition, we provide an answer to the question: Under what conditions are nonvanishing real-valued strongly uniformly locally Lipschitz functions stable under reciprocation in strongly uniformly locally bounded metric spaces?

In Section~3, we investigate the behavior of continuous, Cauchy-continuous, CC-regular, and CBC-regular mappings on cofinally Bourbaki quasi-Cauchy sequences, and derive additional characterizations of cofinally Bourbaki quasi-complete metric spaces and their completions. In addition, we identify necessary and sufficient conditions under which cofinally complete and cofinally Bourbaki complete spaces exhibit this stronger form of completeness. Although these function classes have been fruitfully applied in the study of other generalized completeness properties, their relationship to cofinally Bourbaki quasi-Cauchy sequences proves to be more delicate. In particular, since every sequence in $\mathbb{R}$ is cofinally Bourbaki quasi-Cauchy, the preservation of such sequences by continuous real-valued mappings cannot, on its own, serve to characterize completeness. This subtlety highlights the need for more refined criteria and underscores the novelty of the results obtained in this section.

\section{Characterizations of Cofinally Bourbaki Quasi-Completeness via Lipschitz-Type Functions}

In this section, we investigate the extent to which Lipschitz-type conditions are imposed on real-valued functions can be used to characterize cofinally Bourbaki quasi-complete metric spaces. Our approach builds upon the classical theory of locally Lipschitz, uniformly locally Lipschitz, and Cauchy-Lipschitz functions as developed in \cite{BeerBook, Beer2015, Beer2016}. In this context, we also introduce the class of strongly uniformly locally bounded metric spaces, which play a central role in our main results.

We begin by recalling several standard notions from the Lipschitz function hierarchy.

\begin{definition}[\cite{BeerBook, Beer2015, Beer2016}]
Let $\langle X, d \rangle$ and $\langle Y, \rho \rangle$ be metric spaces, and let $f: X \to Y$ be a function.
\begin{enumerate}
    \item $f$ is said to be \emph{locally Lipschitz} if for each point $x \in X$, there exists $\delta_x > 0$ such that the restriction of $f$ to the ball $S_d(x, \delta_x)$ is Lipschitz. The Lipschitz constant and radius $\delta_x$ may vary with $x$.

    \item $f$ is said to be \emph{Cauchy-Lipschitz} if $f$ is Lipschitz on the range of every Cauchy sequence in $X$.

    \item $f$ is said to be \emph{uniformly locally Lipschitz} if there exists a uniform radius $\delta > 0$ such that for every $x \in X$, the restriction of $f$ to the ball $S_d(x, \delta)$ is Lipschitz.
\end{enumerate}
\end{definition}

The locally Lipschitz, Cauchy-Lipschitz and uniformly locally Lipschitz functions have been proven to be effective tools for characterizing several intermediate completeness properties, including UC-spaces and cofinally complete metric spaces. To extend this framework to the cofinally Bourbaki quasi-complete setting, we now introduce the following notion.

\begin{definition}
A metric space $\langle X, d \rangle$ is said to be \emph{strongly uniformly locally bounded} if there exists $\delta > 0$ such that for each $x \in X$, the $\delta$-chainable component $S^{\infty}_d(x, \delta)$ is a bounded subset of $X$.
\end{definition}

Every bounded metric space, as well as every strongly uniformly locally compact metric space (see \cite{Das2024} for the definition), is strongly uniformly locally bounded. However, the converse does not hold in general. For example, the space of positive integers $\mathbb{N}$, equipped with the subspace metric inherited from the real line, is strongly uniformly locally bounded, yet it is not bounded. On the other hand, the open interval $(0,1) \subseteq \mathbb{R}$, endowed with the standard Euclidean metric, is bounded but fails to be strongly uniformly locally compact.

We now arrive at the first main result of this section, which characterizes cofinally Bourbaki quasi-complete metric spaces in terms of Lipschitz-type mappings. The equivalences below show that cofinally Bourbaki quasi-completeness is precisely determined by the behavior of locally Lipschitz functions under the additional assumption that the space is strongly uniformly locally bounded.

\begin{theorem}\label{uniformly_lipschitz_firsttheorem}
Let $\langle X, d \rangle$ be a metric space. The following conditions are equivalent:
\begin{enumerate}
    \item $X$ is cofinally Bourbaki quasi-complete;
    \item $X$ is strongly uniformly locally bounded and every locally Lipschitz function from $X$ into a metric space $\langle Y, d \rangle$ is uniformly locally chain-Lipschitz;
    \item $X$ is strongly uniformly locally bounded and every real-valued locally Lipschitz function on $X$ is uniformly locally chain-Lipschitz;
    \item $X$ is strongly uniformly locally bounded and the real-valued strongly uniformly locally Lipschitz functions are uniformly dense in $C(X, \mathbb{R})$.
\end{enumerate}
\end{theorem}
\begin{proof}
$(1) \Rightarrow (2)$. Let \( f: X \to Y \) be a locally Lipschitz function. Since \( X \) is cofinally Bourbaki quasi-complete, it follows that \( X \) is cofinally complete. Hence, by \cite[Theorem 30.10]{BeerBook}, the function \( f \) is uniformly locally Lipschitz. Therefore, there exists \( \varepsilon > 0 \) such that for each \( x \in X \), the restriction \( f|_{S_d(x,\varepsilon)} \) is Lipschitz. In particular, \( f \) is bounded on \( S_d(x, \varepsilon) \) for every \( x \in X \).

By \cite[Theorem 4.1]{Ghosh2025}, there exists \( \delta > 0 \) such that for each \( x \in X \), there exist points \( p_1, \dots, p_k \in X \) satisfying
\(
S_d^\infty(x, \delta) \subseteq \bigcup_{i=1}^k S_d\left(p_i, \frac{\varepsilon}{4} \right).
\)
This implies that \( X \) is strongly uniformly locally bounded. Let \( u, v \in S_d^\infty(x, \delta) \). If \( d(u, v) < \frac{\varepsilon}{4} \) and \( u \in S_d(p_i, \frac{\varepsilon}{4}) \) for some \( i \in \{1, \dots, k\} \), then
\(
d(v, p_i) \leq d(v, u) + d(u, p_i) < \varepsilon.
\)
Thus, \( u, v \in S_d(p_i, \varepsilon) \), and since \( f \) is Lipschitz on \( S_d(p_i, \varepsilon) \), it follows that
\(
\rho(f(u), f(v)) \leq L_i d(u, v),
\)
where \( L_i \) is the Lipschitz constant on \( S_d(p_i, \varepsilon) \).

Now suppose \( d(u, v) \geq \frac{\varepsilon}{4} \). Since \( f \) is bounded on \( \bigcup_{i=1}^k S_d(p_i, \varepsilon) \), there exists \( M > 0 \) and a fixed point \( y_0 \in Y \) such that \( f(x) \in S_\rho(y_0, M) \) for all \( x \in \bigcup_{i=1}^k S_d(p_i, \varepsilon) \). Therefore,
\[
\frac{\rho(f(u), f(v))}{d(u, v)} \leq \frac{2M}{d(u, v)} \leq \frac{8M}{\varepsilon}.
\]

Combining both cases, for any \( u, v \in S_d^\infty(x, \delta) \), we obtain
\[
\rho(f(u), f(v)) \leq C d(u, v),
\]
where \( C = \max\{L_i : i = 1, \dots, k\} + \frac{8M}{\varepsilon} \). Hence, \( f \) is Lipschitz on \( S_d^\infty(x, \delta) \), and since this holds uniformly over \( x \in X \), we conclude that \( f \) is uniformly locally chain-Lipschitz on \( X \).

$(2)\Rightarrow(3)$ It is trivial.

$(3)\Rightarrow(4)$ Condition (4) is an immediate consequence of \cite[Theorem 26.3]{BeerBook} since by condition (3), real-valued locally Lipschitz functions on $X$ are uniformly locally chain-Lipschitz.

$(4)\Rightarrow(1)$ Let $f: X\to\mathbb R$ be continuous. By \cite[Theorem 4.1]{Ghosh2025}, it is sufficient to show that $f$ is uniformly component-bounded. By condition (4), we can uniformly approximate each real-valued continuous function by strongly uniformly locally Lipschitz functions. Therefore, we can choose \( \varepsilon > 0 \) such that for each \( x \in X \), the restriction \( f|_{S_d^{\infty}(x,\varepsilon)} \) is Lipschitz and $S_d^{\infty}(x,\varepsilon)$ is bounded. Easily, $f(S_d^{\infty}(x,\varepsilon))$ is bounded for each $x\in X$.
\end{proof}

 The result below provides a characterization of the cofinally Bourbaki quasi-completeness of the completion of a metric space in terms of the behavior of Cauchy-Lipschitz functions defined on the original space. As in the preceding theorem, the characterization is obtained under the assumption that the space is strongly uniformly locally bounded, a condition that ensures sufficient control over the chainable structure of neighborhoods. In the course of this analysis, we also identify when the class of Cauchy-Lipschitz functions coincides with that of strongly uniformly locally Lipschitz functions.

\begin{theorem}\label{completeion_uniformly_lipschitz_theorem}
Let \( \langle X, d \rangle \) be a metric space. The following conditions are equivalent:
\begin{enumerate}
    \item The completion \( \langle \widehat{X}, \widehat{d} \rangle \) of \( \langle X, d \rangle \) is cofinally Bourbaki quasi-complete;
    \item $X$ is strongly uniformly locally bounded in itself and every Cauchy-Lipschitz function from $X$ into a metric space \( \langle Y, \rho \rangle \) is uniformly locally chain-Lipschitz;
    \item $X$ is strongly uniformly locally bounded in itself and every real-valued Cauchy-Lipschitz function on \( X \) is uniformly locally chain-Lipschitz;
    \item $X$ is strongly uniformly locally bounded in itself and the real-valued strongly uniformly locally Lipschitz functions on \( X \) are uniformly dense in \( CC(X, \mathbb{R}) \), the space of real-valued Cauchy-continuous functions on \( X \).
\end{enumerate}
\end{theorem}
\begin{proof}
\((1) \Rightarrow (2)\) Since \( \langle \widehat{X}, \widehat{d} \rangle \) is cofinally Bourbaki quasi-complete, there exists $\delta>0$ such that $S_{\widehat{d}}^{\infty}(x, \delta)$ is bounded for all $x\in \widehat{X}$. Since $S_d^{\infty}(x, \delta)\subseteq S_{\widehat{d}}^{\infty}(x, \delta)$, $X$ is strongly uniformly locally bounded in itself. Let \( f: X \to Y \) be a Cauchy-Lipschitz function, where \( \langle Y, \rho \rangle \) is a metric space. 

Since \( \langle \widehat{X}, \widehat{d} \rangle \) is cofinally complete, by \cite[Proposition 30.11]{BeerBook}, $f$ is uniformly locally Lipschitz function on $X$. Now, we can view \( f \) as a function into the completion \( \langle \widehat{Y}, \widehat{\rho} \rangle \). By \cite[Proposition 27.6]{BeerBook}, \( f \) has a uniformly locally Lipschitz extension \( \widehat{f}: \widehat{X} \to \widehat{Y} \). Now, similar to the proof of \((1) \Rightarrow (2)\) of Theorem \ref{uniformly_lipschitz_firsttheorem}, one can prove that $\widehat{f}$ is uniformly locally chain-Lipschitz on $\widehat{X}$. Hence, the restriction \( f = \widehat{f}|_X \) is uniformly locally chain-Lipschitz on \( X \).


\((2) \Rightarrow (3)\) It is obvious.

\((3) \Rightarrow (4)\) Condition (4) is an immediate consequence of condition (3) and \cite[Theorem 26.8]{BeerBook}.

\((4) \Rightarrow (1)\)  Let \( \widehat{f} \in C(\widehat{X}, \mathbb{R}) \) be given. By condition (7) of \cite[Theorem 4.1]{Ghosh2025}, it suffices to show that \( \widehat{f} \) is uniformly locally component bounded on $\widehat{X}$. 

Let \( f := \widehat{f}|_X \). Since \( \langle \widehat{X}, \widehat{d} \rangle \) is a complete metric space, the function \( \widehat{f} \) is Cauchy continuous, and thus \( f \in CC(X, \mathbb{R}) \). By assumption (4), there exists a uniformly locally chain-Lipschitz function \( g: X \to \mathbb{R} \) such that \(\operatorname{sup}_{x\in X} |f(x) - g(x)|< \infty \).

Choose \( \delta > 0 \) such that \( g \) is Lipschitz on \( S_d^{\infty}(x, \delta) \) and \( S_d^{\infty}(x, \delta) \) is bounded in $X$ for each $x\in X$. Fix \( \widehat{x} \in \widehat{X} \). Then there exists \( x \in X \) such that $\widehat{d}(x, \widehat{x})<\frac{\delta}{3}$. We show that
\(
S_{\widehat{d}}^{\infty}\left( \widehat{x}, \frac{\delta}{3} \right) \subseteq \operatorname{cl}_{\widehat{X}}S_d^{\infty}(x, \delta).
\)
Let $\widehat{y}\in S_{\widehat{d}}^{\infty}\left( \widehat{x}, \frac{\delta}{3} \right)$.Then there exists a finite $\frac{\delta}{3}$-chain $\widehat{y}=\widehat{u}_0\leq \widehat{u}_1\leq\dots\leq\widehat{u}_m=\widehat{x}$ in $\widehat{X}$. Let $\varepsilon>0$ be given. Then there exists $u_j\in X$ such that $\widehat{d}(u_j, \widehat{u}_j)\leq\operatorname{min}\{\frac{\delta}{3},\varepsilon\}$ for each $j=0,\dots,m-1$ and $u_m=x$. Note that for $j=0,\dots,m-1$, $$d(u_j, u_{j+1})=\widehat{d}(u_j, u_{j+1})\leq\widehat{d}(u_j, \widehat{u}_{j+1})+\widehat{d}(\widehat{u}_j, \widehat{u}_{j+1})+\widehat{d}(u_j, \widehat{u}_j)<\delta.$$ Thus $u_0\leq u_1\leq\dots\leq u_m=x$ is a $\delta$-chain in $X$. Clearly, $\widehat{y}\in \operatorname{cl}_{\widehat{X}}S_d^{\infty}(x, \delta)$. 

Since \( g \) is bounded on \( S^{\infty}_d(x, \delta) \), \( f \) is bounded on \( S^{\infty}_d(x, \delta) \). Therefore, \( \widehat{f} \) is bounded on \( S_{\widehat{d}}^{\infty}\left( \widehat{x}, \frac{\delta}{3} \right)\) as \( \widehat{f} \) is continuous. Since \( \widehat{x} \in \widehat{X} \) is arbitrary, we conclude that \( \widehat{f} \) is uniformly locally chain bounded on \( \widehat{X} \).
\end{proof}

The last result also yields a condition under which every Cauchy-continuous real-valued function on $\langle X, d\rangle$ can be uniformly approximated by uniformly locally chain-Lipschitz functions.

\begin{proposition}
    Let $\langle x_n\rangle$ be a cofinally Bourbaki quasi-Cauchy sequence of distinct terms in a metric space $\langle X, d \rangle$. Then there exists a pairwise disjoint family $\{M_j : j \in \mathbb{N}\}$ of infinite subsets of $\mathbb{N}$ such that if $i, \ell \in M_j$, then $x_i$ and $x_\ell$ can be joined by a $\frac 1 j$-chain.
\end{proposition}

\begin{proof}    
Suppose that the sequence $\langle x_n \rangle$ has a Bourbaki quasi-Cauchy subsequence $\langle x_{n_k} \rangle$. Let $N_0 = \{n_k : k \in \mathbb{N}\}$. For each $j \in \mathbb{N}$, choose $m_j \in \mathbb{N}$ such that for all $i > \ell \geq m_j$ with $\{i, \ell\} \subseteq N_0$, the points $x_i$ and $x_\ell$ can be joined by a $\frac{1}{j}$-chain.

Partition $N_0$ into countably many infinite subsets $\{K_j : j \in \mathbb{N}\}$. Define
\[
M_j := \{n \in K_j : n \geq m_j\}.
\]
Then $\{M_j : j \in \mathbb{N}\}$ is a pairwise disjoint family of infinite subsets of $\mathbb{N}$ satisfying the desired conditions.

Now suppose that $\langle x_n \rangle$ has no Bourbaki quasi-Cauchy subsequence. Then, we can choose an infinite set $M_1 \subseteq \mathbb{N}$ such that for all $i, \ell \in M_1$, the points $x_i$ and $x_\ell$ can be joined by a $1$-chain. Since the sequence $\langle x_n \rangle$ has no Bourbaki quasi-Cauchy subsequence, the set $\{x_i : i \in M_1\}$ is not qC-precompact, by \cite[Theorem~2.2]{Adhikary2024}.

Hence, by passing to an infinite subset of $M_1$, there exists $\varepsilon_1 < \frac{1}{2}$ such that for all $i, \ell \in M_1$, the points $x_i$ and $x_\ell$ can be joined by a $1$-chain but not by any $\varepsilon_1$-chain.

Next, choose an infinite set $M_2 \subseteq \mathbb{N}$ such that for all $i, \ell \in M_2$, the points $x_i$ and $x_\ell$ can be joined by an $\varepsilon_1$-chain. By construction, the sets $\{x_i : i \in M_1\}$ and $\{x_i : i \in M_2\}$ have at most one common point.

Again, since the set $\{x_i : i \in M_2\}$ is not qC-precompact, we may pass to an infinite subset of $M_2$ and find $\varepsilon_2 < \frac{1}{3}$ such that for all $i, \ell \in M_2$, the points $x_i$ and $x_\ell$ can be joined by a $\frac{1}{2}$-chain but not by any $\varepsilon_2$-chain.

Continuing inductively in this manner, we construct a sequence of infinite subsets $\{M_j\}_{j \in \mathbb{N}}$ of $\mathbb{N}$ such that for each $j \in \mathbb{N}$:
\begin{itemize}
    \item The sets $\{x_i : i \in M_j\}$ are pairwise disjoint across $j$.
    \item There exists $\varepsilon_j < \frac{1}{j+1}$ such that for all $i, \ell \in M_j$, the points $x_i$ and $x_\ell$ can be joined by a $\frac{1}{j}$-chain but not by any $\varepsilon_j$-chain.
\end{itemize}
This completes the construction of the desired family $\{M_j : j \in \mathbb{N}\}$.
\end{proof}

Our next result addresses the following question: Under what conditions are nonvanishing real-valued strongly uniformly locally Lipschitz functions stable under reciprocation in strongly uniformly locally bounded metric spaces?

\begin{theorem}
Let $\langle X, d \rangle$ be a metric space. The following conditions are equivalent:
\begin{enumerate}
    \item The metric space $\langle X, d \rangle$ is cofinally Bourbaki quasi-complete.
    \item $X$ is strongly uniformly locally bounded and the reciprocal of every nonvanishing uniformly locally chain-Lipschitz real-valued function on $X$ is uniformly locally chain-Lipschitz.
    \item $X$ is strongly uniformly locally bounded and the reciprocal of every nonvanishing Lipschitz real-valued function on $X$ is uniformly locally chain-Lipschitz.
\end{enumerate}
\end{theorem}
\begin{proof}
$(1)\Rightarrow (2)$ Since every uniformly locally chain-Lipschitz function is locally Lipschitz, condition $(2)$ follows from Theorem~\ref{uniformly_lipschitz_firsttheorem} together with the fact that the reciprocal of every nonvanishing locally Lipschitz function on a metric space is itself locally Lipschitz.

$(2)\Rightarrow (3)$ It is obvious.

$(3)\Rightarrow (1)$ Suppose to the contrary that $X$ is not cofinally Bourbaki quasi-complete. Then there exists a cofinally Bourbaki quasi-Cauchy sequence $\langle x_n\rangle$ in $X$ that does not cluster. By passing to a subsequence, we may assume that $\langle x_n\rangle$ consists of distinct terms. Additionally, without loss of generality, let $\mathbb{N} = \bigcup_{j \in \mathbb{N}} M_j$ be a partition of the natural numbers into infinite subsets $M_j$ such that whenever $k\in M_j$, there exists $p_j\in X$ such that $x_k\in S_d^{\infty}(p_j,\frac 1 j)$ and $S_d^{\infty}(p_j,\frac 1 j)$ is bounded.

For each positive integer $n$, define
\(
\delta_n := d\left(x_n, \{x_k : k \neq n\}\right).
\)

Define a function $f : \{x_n : n \in \mathbb{N}\} \to (0, \infty)$ by
\[
f(x_n) := \frac{\delta_n}{n}.
\]
Then, for $n \neq k$, we compute:
\[
|f(x_n) - f(x_k)| \leq \frac{\delta_n}{n} + \frac{\delta_k}{k} \leq \delta_n + \delta_k \leq 2 d(x_n, x_k).
\]
Thus, $f$ is $2$-Lipschitz on its domain.

Since $\{x_n : n \in \mathbb{N}\}$ is closed in $X$, by \cite[Lemma 2.2]{Beer2020}, $f$ admits a $2$-Lipschitz extension $g : X \to (0,\infty)$ that is positive and defined on all of $X$.

However, for each $j \in \mathbb{N}$, the set $\{x_n : n \in M_j\}$ has a finite diameter, yet $1/g$ is unbounded on this set. Indeed, since $g(x_n) = \delta_n/n$, we have
\[
\frac{1}{g(x_n)} = \frac{n}{\delta_n} \to \infty \quad \text{as } n \to \infty \text{ in } M_j.
\]
Thus, while $1/g$ is locally Lipschitz, it fails to be strongly uniformly locally chain-Lipschitz, since on each set $\{x_n : n \in M_j\}$, it is unbounded. 
\end{proof}

We conclude this section with an elementary but useful observation concerning the algebraic structure of the space of strongly uniformly locally Lipschitz real-valued functions. The proof is left to the interesting reader.

\begin{theorem}
Let $\langle X, d \rangle$ be a metric space. Then the family of real-valued strongly uniformly locally Lipschitz functions on $X$ is closed under pointwise multiplication if and only if $X$ is strongly uniformly locally bounded.
\end{theorem}

\section{Characterizations of Cofinally Bourbaki Quasi-Completeness via Continuous-Type Functions}

In this section, we present several characterizations of cofinally Bourbaki quasi-complete metric spaces based on the behavior of real-valued functions acting on cofinally Bourbaki quasi-Cauchy sequences. In particular, the result below parallels classical stability results for function classes under reciprocal operations.

\begin{theorem}\label{reciprocal}
Let $\langle X, d\rangle$ be a metric space. Then the following statements are equivalent:
\begin{enumerate}
    \item $X$ is cofinally Bourbaki quasi-complete;
    \item If $f : X \to \mathbb{R}$ is a continuous function that maps every cofinally Bourbaki quasi-Cauchy sequence in $X$ to a cofinally Bourbaki-Cauchy sequence in $\mathbb{R}$, and if $f$ is nowhere zero, then the reciprocal function $\frac{1}{f}$ is also continuous and maps cofinally Bourbaki quasi-Cauchy sequences in $X$ to cofinally Bourbaki-Cauchy sequences in $\mathbb{R}$.
\end{enumerate}
\end{theorem}

\begin{proof}
$(1)\Rightarrow (2)$ Let $g: X \to \mathbb{R}$ be any continuous function and let $\langle x_n\rangle$ be a cofinally Bourbaki quasi-Cauchy sequence in $X$. Since $X$ is cofinally Bourbaki quasi-complete, $\langle x_n\rangle$ has a convergent subsequence $\langle x_{n_k}\rangle$. Thus, $\langle g(x_{n_k})\rangle$ is a convergent subsequence of $\langle g(x_n)\rangle$. Hence $\langle g(x_n)\rangle$ is cofinally Bourbaki-Cauchy. Since we just proved every real-valued continuous function on $X$ maps every cofinally Bourbaki quasi-Cauchy sequence in $X$ to a cofinally Bourbaki-Cauchy sequence in $\mathbb{R}$, condition (2) follows.

$(2)\Rightarrow (1)$ Suppose that condition (1) fails. Then there exists a cofinally Bourbaki quasi-Cauchy sequence $\langle x_n\rangle$ of distinct points in $X$ that have no cluster point. Let $A := \{x_n : n \in \mathbb{N}\}$. Then $A$ is closed and discrete. Define a function $g : A \to \mathbb{R}$ by $g(x_n) := \frac{1}{2^n}$ for all $n \in \mathbb{N}$. Clearly, $g$ is continuous and it maps cofinally Bourbaki quasi-Cauchy sequences to cofinally Bourbaki-Cauchy sequences. Then there exists a continuous extension $f : X \to (0,1)$ of $f$. Now, since $(0,1)$ is Bourbaki-bounded in $\mathbb{R}$, every sequence in $(0,1)$ is cofinally Bourbaki–Cauchy \cite{Garrido2014}. Hence, $f$ maps every cofinally Bourbaki quasi-Cauchy in $X$ to cofinally Bourbaki-Cauchy sequences in $\mathbb{R}$. However, $f$ is nowhere zero and the sequence $\langle x_n\rangle$ is cofinally Bourbaki quasi–Cauchy in $X$, whereas the sequence $\langle \frac{1}{f(x_n)}\rangle = \langle n\rangle$ is not cofinally Bourbaki–Cauchy in $\mathbb{R}$, which is a contradiction.
\end{proof}

We present a characterization of cofinally Bourbaki quasi-complete metric spaces in terms of the behavior of continuous functions acting on cofinally Bourbaki quasi-Cauchy sequences. While this may appear analogous to the classical result that a metric space is cofinally complete if and only if every real-valued continuous function is CC-regular. However, the present setting is significantly more subtle because every sequence in \( \mathbb{R} \) is cofinally Bourbaki quasi-Cauchy, and thus it is not sufficient to require that real-valued continuous functions merely preserve such sequences. The result below therefore identifies a refined condition that fully captures the completeness structure of the space.

\begin{theorem}\label{continuous}
 Let $\langle X, d\rangle $ be a metric space. Then the following statements are equivalent:
\begin{enumerate}
    \item\(X\) is cofinally Bourbaki quasi-complete;
    \item Every continuous function $f$ from $X$ to any other metric space $\langle Y, \rho\rangle$ maps every cofinally Bourbaki quasi-Cauchy sequence in \( X \) to a sequence in \( Y \) that has a Cauchy subsequence;
    \item Every continuous function $f$ from $X$ to any other metric space $\langle Y, \rho\rangle$ maps every cofinally Bourbaki quasi-Cauchy in $X$ to a cofinally Cauchy sequence in $Y$;
     \item Every continuous function $f$ from $X$ to any other metric space $\langle Y, \rho\rangle$ maps every cofinally Bourbaki quasi-Cauchy in $X$ to a cofinally  Bourbaki-Cauchy sequence in $Y$;
    \item Every real-valued continuous function on \(X\) maps every cofinally Bourbaki quasi-Cauchy in $X$ to a cofinally Bourbaki-Cauchy sequence in $\mathbb R$.
\end{enumerate}
\end{theorem}
\begin{proof}
\((1) \Rightarrow (2)\) This follows immediately from the fact that every cofinally Bourbaki quasi-Cauchy sequence in \( X \) has a cluster point.

The implications \((2) \Rightarrow (3)\) and \((3) \Rightarrow (4)\) are immediate, where as the implication \((4) \Rightarrow (5)\) is straightforward.

\((5) \Rightarrow (1)\) Suppose that condition (1) fails. Then there exists a cofinally Bourbaki quasi-Cauchy sequence $\langle x_n\rangle$ of distinct points in $X$ that have no cluster point. Let $A := \{x_n : n \in \mathbb{N}\}$. Then $A$ is closed and discrete. Define a function $f : A \to \mathbb{R}$ by $f(x_n) := n^2$ for all $n \in \mathbb{N}$. Clearly, $f$ is continuous. Since $A$ is closed, there exists a continuous extension $F : X \to \mathbb R$ of $f$. However, $F$ does not map the cofinally Bourbaki quasi-Cauchy $\langle x_n\rangle$ to a cofinally Bourbaki-Cauchy sequence, which yields a contradiction. 
\end{proof}
\begin{corollary}
Let $\langle X, d\rangle $ be a metric space and let $\langle Y, \|,\|\rangle$ be a not qC-precompact normed linear space. Then \(X\) is cofinally Bourbaki quasi-complete if and only if every continuous function from $X$ to $Y$ maps cofinally Bourbaki quasi-Cauchy sequences to cofinally Bourbaki quasi-Cauchy sequences.
\end{corollary}
\begin{proof}
If \(X\) is cofinally Bourbaki quasi-complete, then by condition~(4) of Theorem~\ref{continuous}, every continuous function from \(X\) to a metric space \(Y\) maps cofinally Bourbaki quasi-Cauchy sequences to cofinally Bourbaki-Cauchy sequences. Since every cofinally Bourbaki-Cauchy sequence is also cofinally Bourbaki quasi-Cauchy, the necessary condition holds.

Conversely, suppose \(X\) is not cofinally Bourbaki quasi-complete. Then there exists a cofinally Bourbaki quasi-Cauchy sequence \(\langle x_n \rangle\) of distinct points in \(X\) that has no cluster point. Let \(A := \{x_n : n \in \mathbb{N}\}\). Then \(A\) is closed and discrete. Define a function \(f : A \to Y\) by \(f(x_n) := y_n\) for all \(n \in \mathbb{N}\), where \(\langle y_n \rangle\) is a sequence in \(Y\) that is not cofinally Bourbaki quasi-Cauchy. Clearly, \(f\) is continuous. Since \(A\) is closed and \(Y\) is a normed linear space, by the Dugundji extension theorem~\cite[pp.~188]{Dugundji}, there exists a continuous extension \(F : X \to Y\) of \(f\). However, \(F\) does not map the cofinally Bourbaki quasi-Cauchy sequence \(\langle x_n \rangle\) to a cofinally Bourbaki-Cauchy sequence in \(Y\), which yields a contradiction.  
\end{proof}    

We now turn to the setting of completions and consider how cofinally Bourbaki quasi-completeness of the completion \(\widehat{X}\) may be described in terms of the behavior of Cauchy-continuous functions defined on the original space \(X\). 

\begin{theorem}\label{cauchycontinuous_completion}
 Let $\langle X, d\rangle $ be a metric space and $(\widehat{X}, \widehat{d})$ be its completion. Then the following statements are equivalent:
    \begin{enumerate}
    \item $\widehat{X}$ is cofinally Bourbaki quasi-complete;
    \item Every Cauchy-continuous function $f$ from $X$ to any other metric space $\langle Y, \rho\rangle$ maps every cofinally Bourbaki quasi-Cauchy in $X$ to a sequence in \( Y \) that has a Cauchy subsequence;
    \item Every Cauchy-continuous function $f$ from $X$ to any other metric space $\langle Y, \rho\rangle$ maps every cofinally Bourbaki quasi-Cauchy in $X$ to a cofinally Cauchy sequence in $Y$;
    \item Every Cauchy-continuous function $f$ from $X$ to any other metric space $\langle Y, \rho\rangle$ maps every cofinally Bourbaki quasi-Cauchy in $X$ to a cofinally Bourbaki-Cauchy sequence in $Y$;
    \item Every real-valued Cauchy-continuous function on \(X\) maps every cofinally Bourbaki quasi-Cauchy in $X$ to a cofinally Bourbaki-Cauchy sequence in $\mathbb R$;
    \item Every cofinally Bourbaki quasi–Cauchy sequence in \(X\) has a Cauchy subsequence.
\end{enumerate}
\end{theorem}
\begin{proof}
 \((1) \Rightarrow (2)\) It is immediate from the fact that every cofinally Bourbaki quasi-Cauchy in $X$ has a Cauchy subsequence.  

 The implications \((2) \Rightarrow (3)\) and \((3) \Rightarrow (4)\) are immediate, where as the implication \((4) \Rightarrow (5)\) is straightforward.

 \((5) \Rightarrow (6)\) Suppose that condition~\((6)\) fails. Then there exists a cofinally Bourbaki quasi-Cauchy sequence \(\langle x_n \rangle\) of distinct terms in \(X\) that has no Cauchy subsequence. Let \(A = \{x_n : n \in \mathbb{N}\}\). Define \(f : A \to \mathbb{R}\) by \(f(x_n) = n\) for all \(n \in \mathbb{N}\). Clearly, \(f\) is Cauchy-continuous on \(A\). By \cite[Theorem 3.4]{BeerBook}, \(f\) admits a Cauchy-continuous extension \(F : X \to \mathbb{R}\). However, \(F\) does not map the cofinally Bourbaki quasi-Cauchy sequence \(\langle x_n \rangle\) to a cofinally Bourbaki–Cauchy sequence, which yields a contradiction.
 \((6) \Rightarrow (1)\) Let $\langle \widehat{x}_n\rangle$ be a cofinally Bourbaki quasi-Cauchy sequence in $\widehat{X}$. For each $n \in \mathbb{N}$, choose $x_n \in X$ such that $\widehat{d}(\widehat{x}_n, x_n) \leq \frac{1}{n}$. Let $\varepsilon > 0$ be given. Then there exists an infinite subset $N_\varepsilon \subseteq \mathbb{N}$ such that for all $i, \ell \in N_\varepsilon$, the points $\widehat{x}_i$ and $\widehat{x}_\ell$ can be joined by an $\frac{\varepsilon}{3}$-chain in $\widehat{X}$. Then there exists an infinite subset $M_\varepsilon$ of $N_\varepsilon$ such that for all $i, \ell \in M_\varepsilon$, the points $x_i$ and $x_\ell$ can be joined by an $\varepsilon$-chain in $X$. Hence, the sequence $\langle x_n \rangle$ is cofinally Bourbaki quasi-Cauchy in $X$, and by assumption, it has a Cauchy subsequence. Hence, the sequence $\langle x_n \rangle$ has a cluster point in $\widehat{X}$. It follows that $\langle \widehat{x}_n \rangle$ also clusters in $\widehat{X}$.
\end{proof}

The next result characterizes cofinally Bourbaki quasi-completeness in terms of cofinally Bourbaki completeness and the behavior of CBC-regular functions. In particular, it identifies a precise condition under which cofinally Bourbaki completeness extends to cofinally Bourbaki quasi-completeness. 

\begin{theorem}\label{cbqc_cbc}
Let $\langle X, d\rangle $ be a metric space. Then the following statements are equivalent:
\begin{enumerate}
    \item $X$ is cofinally Bourbaki quasi-complete;
    \item $X$ is cofinally Bourbaki-complete and every CBC-regular function from $X$ to any metric space $Y$ maps cofinally Bourbaki quasi-Cauchy sequences to cofinally Bourbaki-Cauchy sequences;
    \item $X$ is cofinally Bourbaki-complete and every real-valued CBC-regular function on $X$ maps cofinally Bourbaki quasi-Cauchy sequences to cofinally Bourbaki-Cauchy sequences;
    \item $X$ is cofinally Bourbaki-complete and every cofinally Bourbaki quasi–Cauchy sequence in $X$ is cofinally Bourbaki-Cauchy.
\end{enumerate}
\end{theorem}

\begin{proof}
 $(1)\Rightarrow (2)$ Let $f : X \to (Y, \rho)$ be a CBC-regular function, and let $\langle x_n\rangle $ be a cofinally Bourbaki quasi–Cauchy sequence in $X$. Since $X$ is cofinally Bourbaki quasi-complete, there exists a subsequence $\langle x_{n_k}\rangle $ of $\langle x_n\rangle $ that converges in $X$. In particular, $\langle x_{n_k}\rangle $ is cofinally Bourbaki-Cauchy, and hence the sequence $\langle f(x_{n_k})\rangle $ is cofinally Bourbaki-Cauchy in $Y$ by our assumption. Consequently, $f$ maps cofinally Bourbaki quasi-Cauchy sequences to cofinally Bourbaki-Cauchy sequences.

 $(2)\Rightarrow (3)$ It is obvious.

$(3)\Rightarrow (4)$ Suppose there exists a cofinally Bourbaki quasi–Cauchy sequence $\langle x_n\rangle $ of distinct points in $X$ which is not cofinally Bourbaki-Cauchy. Define a function $f : X \to \mathbb{R}$ by
\[
f(x) = 
\begin{cases}
2^n, & \text{if } x = x_n \text{ for some } n \in \mathbb{N}, \\
0, & \text{otherwise}.
\end{cases}
\]
Clearly, $f$ is CBC-regular. However, the sequence $\langle x_n\rangle$ is cofinally Bourbaki quasi–Cauchy in $X$, whereas the sequence $\langle {f(x_n)}\rangle = \langle 2^n\rangle$ is not cofinally Bourbaki–Cauchy in $\mathbb{R}$, which is a contradiction.

$(4)\Rightarrow (1)$ It is immediate.
    
\end{proof}

The following theorem can be proved similarly:

\begin{theorem}
Let $\langle X, d\rangle $ be a metric space. Then the following statements are equivalent:
\begin{enumerate}
    \item $X$ is cofinally Bourbaki quasi-complete;
    \item $X$ is cofinally complete and every CC-regular function from $X$ to any metric space $Y$ maps cofinally Bourbaki quasi-Cauchy sequences to cofinally Cauchy sequences;
    \item $X$ is cofinally complete and every real-valued CC-regular function on $X$ maps cofinally Bourbaki quasi-Cauchy sequences to cofinally Cauchy sequences;
    \item $X$ is cofinally complete and every cofinally Bourbaki quasi–Cauchy sequence in $X$ is cofinally Cauchy.
\end{enumerate}
\end{theorem}

\begin{theorem}
 Let $\langle X, d\rangle $ be a metric space and $(\widehat{X}, \widehat{d})$ be its completion. Then the following statements are equivalent:
\begin{enumerate}
    \item $\widehat{X}$ is cofinally Bourbaki quasi-complete;
    \item $\widehat{X}$ is cofinally Bourbaki-complete and every CBC-regular function from $X$ to any other metric space $(Y, \rho)$ maps cofinally Bourbaki quasi-Cauchy sequences to cofinally Bourbaki-Cauchy sequences;
    \item $\widehat{X}$ is cofinally Bourbaki-complete and every real-valued CBC-regular function on $X$ maps cofinally Bourbaki quasi-Cauchy sequences to cofinally Bourbaki-Cauchy sequences;
    \item $\widehat{X}$ is cofinally Bourbaki-complete and every cofinally Bourbaki quasi–Cauchy sequence in $X$ is cofinally Bourbaki-Cauchy.
\end{enumerate}
\end{theorem}
\begin{proof}
 $(1)\Rightarrow (2)$ Since $\widehat{X}$ is cofinally Bourbaki quasi-complete, it is cofinally Bourbaki-complete. The rest follows from the fact that every cofinally Bourbaki quasi-Cauchy sequence in $X$ has a cofinally Bourbaki-Cauchy subsequence by (1).
 
  $(2)\Rightarrow (3)$ It is obvious.  

  $(3)\Rightarrow(4)$ The proof proceeds analogously to the implication \((3) \Rightarrow (4)\) of Theorem~\ref{cbqc_cbc}.

  $(4)\Rightarrow (1)$ Let \(\langle \widehat{x}_n \rangle\) be a cofinally Bourbaki quasi-Cauchy sequence in \(\widehat{X}\). For each \(n \in \mathbb{N}\), choose \(x_n \in X\) such that \(\widehat{d}(\widehat{x}_n, x_n) \leq \frac{1}{n}\). Then the sequence \(\langle x_n \rangle\) is cofinally Bourbaki quasi-Cauchy in \(X\) (see the proof of \((6) \Rightarrow (1)\) in Theorem~\ref{cauchycontinuous_completion}), and by assumption, it is cofinally Bourbaki-Cauchy. Since \(\widehat{X}\) is cofinally Bourbaki-complete, the sequence \(\langle x_n \rangle\) has a cluster point in \(\widehat{X}\). It follows that \(\langle \widehat{x}_n \rangle\) also clusters in \(\widehat{X}\).
\end{proof}

Similarly, we have the following:

\begin{theorem}
 Let $\langle X, d\rangle $ be a metric space and $(\widehat{X}, \widehat{d})$ be its completion. Then the following statements are equivalent:
\begin{enumerate}
    \item $\widehat{X}$ is cofinally Bourbaki quasi-complete;
    \item $\widehat{X}$ is cofinally complete and every CC-regular function from $X$ to any other metric space $\langle Y, \rho\rangle$ maps cofinally Bourbaki quasi-Cauchy sequences to cofinally Cauchy sequences;
    \item $\widehat{X}$ is cofinally complete and every real-valued CC-regular function on $X$ maps cofinally Bourbaki quasi-Cauchy sequences to cofinally Cauchy sequences;
    \item $\widehat{X}$ is cofinally complete and every cofinally Bourbaki quasi–Cauchy sequence in $X$ is cofinally Cauchy.
\end{enumerate}
\end{theorem}

\section*{Acknowledgment:} The author thanks the referee for helpful suggestions and for the time spent in reviewing this work.

\end{document}